\documentclass[12pt]{amsart}
\usepackage{amsmath}
\usepackage{amsfonts}
\usepackage{amssymb}
\usepackage{subfigure}
\usepackage{graphicx}

\numberwithin{figure}{section}

\newtheorem{lemma}{Lemma}[section]

\newtheorem{definition}[lemma]{Definition}

\newtheorem{proposition}[lemma]{Proposition}

\newtheorem{theorem}[lemma]{Theorem}
\numberwithin{equation}{section}

\newcommand{\func}[1]{\mathrm{#1}}

\newcommand{\Rdst}{{\mathbb{R}^d}}
\newcommand{\Rst}{{\mathbb{R}}}

\newcommand{\Rtdst}{{\mathbb{R}^{2d}}}

\newcommand{\norm}[1]{\lVert#1\rVert}

\newcommand{\ip}[2]{\ensuremath{\left<#1,#2\right>}}

\newcommand{\abs}[1]{\ensuremath{\left| #1 \right| }}

\newcommand{\newHz}{{\mathsf{h_0}}}
\newcommand{\newHk}{{\mathsf{h_k}}}

\begin{document}
\title[Sharp rates for accumulated spectrograms]{Sharp
rates of convergence for accumulated spectrograms}
\author{Lu\'{\i}s Daniel Abreu}
\address{Acoustics Research Institute, Austrian Academy of Science,
Wohllebengasse 12-14, A-1040, Vienna Austria}
\author{Jo\~ao Pereira}
\address{Program in Applied and Computational Mathematics, Princeton University, New Jersey 08544,
USA}
\author{Jos\'e Luis Romero}
\address{Acoustics Research Institute, Austrian Academy of Science,
Wohllebengasse 12-14, A-1040, Vienna Austria}
\thanks{L. D. A. was supported by the Austrian Science Fund (FWF): START-project
FLAME (”Frames and Linear Operators for Acoustical Modeling and Parameter Estimation”, Y 551-N13).
J. M. P. was partially supported by NSF CMMI-1310173, AFOSR FA9550-12-1-0317, and NSF CAREER Award 
CCF-1552131. J. L. R. gratefully acknowledges support from the Austrian Science Fund (FWF): P 29462 - N35.}

\begin{abstract}
We investigate an inverse problem in time-frequency localization:
the approximation of the symbol of a time-frequency localization operator from
partial spectral information by the method of accumulated
spectrograms (the sum of the spectrograms corresponding to large
eigenvalues). We derive a sharp bound for the rate of convergence of the
accumulated spectrogram, improving on recent results.
\end{abstract}

\maketitle

\section{Introduction}

In this article we obtain sharp rates of convergence for the approximation
of the symbol of a time-frequency filter from (phaseless) measurements of its
eigenspectrograms, using the method of accumulated spectrograms. The method
is within the realm of inverse problems of time-frequency localization,
where one aims to recover a localization operator from partial spectral
information.

\subsection{Time-frequency localization}

In several branches of signal processing, signals change their frequency
properties over time. It is the case of acoustical signals, such as music,
where frequency variation is perceived as melody. Since the signal's Fourier
transform provides frequency information without localization in time, it is
often preferable to represent a signal simultaneously in the time and
frequency domain. \emph{The short-time Fourier transform} of a function
$f:{\mathbb{R}^{d}}\rightarrow {\mathbb{C}}$ is defined with the aid of a
\emph{window} function $g:{\mathbb{R}^{d}}\rightarrow {\mathbb{C}}$ as
follows:
\begin{equation*}
V_{g}f(z)=\int_{\mathbb{R}^{d}}f(t)\overline{g(t-x)}e^{-2\pi i\xi
t}dt,\qquad z=(x,\xi )\in {\mathbb{R}^{2d}}\text{.}
\end{equation*}
The \emph{spectrogram of} $f$, defined as
\begin{equation*}
\func{Spec}(f)(x,\xi ):=\left\vert V_{g}f(x,\xi )\right\vert ^{2}\text{,}
\end{equation*}
measures the intensity of the contribution to $f$ of the frequency $\xi $
near $x$. Time concentration for $f$ corresponds to decay of $\func{Spec}(f)(x,\xi )$ in $x$, frequency concentration 
for $f$ corresponds to decay of
$\func{Spec}(f)(x,\xi )$ in $\xi $, and simultaneous concentration in time
and frequency corresponds to decay of $\func{Spec}(f)(x,\xi )$ in $(x,\xi )$.\footnote{The spectrogram depends on the underlying window $g$; when we need to stress
this dependence we write $\func{Spec}_g(f)$.}

Since a signal can only be observed and processed when concentrated within a
bounded region of the time-frequency plane, a common practice in signal processing is to use a
time-frequency filter that selects the portion of the spectrogram of a signal that is mostly
concentrated on a given domain.

\subsection{Localization operators}

As a mathematical model of a time-frequency filter, Daubechies suggested an
analogy to the Landau-Pollack-Slepian theory of prolate spheroidal functions
\cite{lapo61,lapo62,posl61,la67,Slepian}. This led to the following notion of \emph{localization operator}
${H_{\Omega }}$ acting on a signal $f$. Given a
compact set $\Omega \subseteq {\mathbb{R}^{2d}}$ and a function $f:{\mathbb{R}^{d}}\rightarrow
{\mathbb{C}}$, let
\begin{equation}
{H_{\Omega }}f(t)=\int_{{\mathbb{R}^{2d}}}1_{\Omega }(x,\xi )V_{g}f(x,\xi
)g(t-x)e^{2\pi i\xi t}dxd\xi \text{,}\qquad t\in {\mathbb{R}^{d}}\text{.}
\label{localization}
\end{equation}
The indicator function $1_{\Omega }(x,\xi )$ is called the symbol of ${H_{\Omega }}$. The
spectrogram of ${H_{\Omega }}f$ is an approximation of $1_{\Omega } \cdot \func{Spec}(f)$ - while
$1_{\Omega } \cdot \func{Spec}(f)$ does not in general correspond to the spectrogram of any
signal.

More generally, it is usual to consider time-frequency localization operators associated with
a general symbol $m \in L^\infty(\Rtdst)$
\begin{align}
\label{eq_tf_loc}
{H_{m}}f(t)=\int_{{\mathbb{R}^{2d}}}m(x,\xi )V_{g}f(x,\xi
)g(t-x)e^{2\pi i\xi t}dxd\xi \text{,}\qquad t\in {\mathbb{R}^{d}}\text{.}
\end{align}
These operators have been studied from the perspective of pseudodifferential calculus
\cite{herato94, herato97, cogr03, teo}.

If $\Omega \subseteq \Rtdst$ is compact, then ${H_{\Omega }}$ is a compact and positive operator on
$L^{2}(\mathbb{R}^{d})$ \cite{bocogr04,cogr03,doro14,ro12}. Hence ${H_{\Omega }}$
can be diagonalized as
\begin{equation*}
{H_{\Omega }}f=\sum_{k\geq 1}{\lambda _{k}^{\Omega }}\left\langle
f,h_{k}^{\Omega }\right\rangle h_{k}^{\Omega },\qquad f\in {L^{2}({\mathbb{R}
^{d}})}\text{,}
\end{equation*}
where $\left\{ {\lambda _{k}^{\Omega }}:k\geq 1\right\} $ are the non-zero
eigenvalues of ${H_{\Omega }}$ ordered non-increasingly and $\left\{
h_{k}^{\Omega }:k\geq 1\right\} $ are the corresponding orthonormal
eigenfunctions. The quality of ${H_{\Omega }}$ as a simultaneous cut-off in
the time-frequency variables can be described by its spectral properties,
because the ($L^{2}$-normalized)
eigenfunctions of $\left\{ h_{k}^{\Omega }:k\geq 1\right\} $ of ${H_{\Omega }
}$ maximize the time-frequency concentration of the spectrogram within $\Omega $. Indeed, since
\begin{equation*}
\left\langle H_{\Omega }f,f\right\rangle =\int_{\Omega }\left\vert
V_{g}f(x,\xi )\right\vert ^{2}dxd\xi \text{,}
\end{equation*}
the min-max lemma for self-adjoint operators gives:
\begin{equation}
\label{eq_mm}
\lambda _{k}^{\Omega }=\max \left\{ \int_{\Omega }\left\vert V_{g}f(x,\xi
)\right\vert ^{2}dxd\xi :\left\Vert f\right\Vert =1,f\perp h_{1}^{\Omega
},...,h_{k-1}^{\Omega }\right\} \text{.}
\end{equation}
Thus, the eigenfunctions $h^\Omega_k$ have short-time Fourier transforms that are optimally
concentrated on the target domain $\Omega$ in the $L^2$ sense. Time-frequency concentration
can also be considered with respect to other metrics, and the fundamental results
are contained in various \emph{uncertainty principles} - see e.g.
\cite{li90-1,rito12,rito13,oltokr13, grma13}.

\subsection{Inverse problems in time-frequency localization}

The spectral problem of time-frequency localization presented in the
previous section consists in describing the eigenfunctions and eigenvalues
of ${H_{\Omega }}$. The corresponding inverse problem
consist in recovering the different ingredients of ${H_{\Omega }}$ (the
window $g$ or the domain $\Omega $) from (partial) spectral information.
Thus, the inverse TF localization problem is a special case of the more general problem of
channel identification \cite{ka62,be69,pf08-1, pfwa06}.

When $g$ is a Gaussian window and $\Omega $ is a disk, the eigenfunctions
of ${H_{\Omega }}$ are Hermite functions, and the corresponding eigenvalues
have an explicit expression (see Section \ref{sec_ex} for
more details). For the corresponding inverse problem, the following was
established in \cite{abdo12}.

\begin{theorem}[\protect\cite{abdo12}]
\label{th_abdo} Let $g(t) := 2^{1/4} e^{-\pi t^2}$, $t \in {\mathbb{R}}$, be
the one-dimensional Gaussian and let $\Omega \subseteq {\mathbb{R}}^2$ be
compact and simply connected. If one of the eigenfunctions of ${H_\Omega}$
is a Hermite function, then $\Omega$ is a disk centered at $0$.
\end{theorem}

Hence, for a Gaussian window $g$, the localization domain $\Omega $ is
completely determined by the information that $\Omega \subseteq {\mathbb{R}^{2d}}$ is a simply
connected set with given measure $|\Omega |$ and that
one of the eigenfunctions of ${H_{\Omega }}$ is a Hermite function. However,
the stylized assumptions of the Theorem \ref{th_abdo} restrict its use in real
applications. First, the restriction on the window $g$ is significant: the
result holds only for Gaussian windows. It is not clear at all how to adapt the
proofs in \cite{abdo12} to more general situations, since they depend on one
variable complex analysis methods, which only apply to Gaussian windows \cite{AscBru}. In addition,
Theorem \ref{th_abdo} does not offer numerical stability:
from the information that the eigenmodes of a time-frequency localization
operator ${H_{\Omega }}$ look approximately like Hermite functions, one
cannot conclude that the localization domain $\Omega $ is approximately a
disk.

A more feasible approach to the approximate recovery of the localization
domain from partial spectral information has been developed in \cite{AGR16},
based on the concept of accumulated spectrogram. It provides a method for
approximating the symbol of a localization operator from the intensities of
the time-frequency representations of its eigenfunctions. Here,
the chosen window $g$ plays only a minor role.

\subsection{Symbol retrieval and the accumulated spectrogram}

Suppose we can measure the spectrograms of the first eigenfunctions
$h_{1}^{\Omega },\ldots ,h_{N}^{\Omega }$ of the eigenvalue problem associated with ${H_{\Omega
}}$ (that is, the eigenfunctions with higher time-frequency energy within the domain $\Omega$ - cf.
\eqref{eq_mm}). The following is the central notion of the article.

\begin{definition}
\label{def}
Let $A_{\Omega }$ be the smallest integer greater than or equal to $\left\vert \Omega
\right\vert $. The \emph{accumulated spectrogram} associated with $g$ and $\Omega $ is:
\begin{equation}
\rho _{g,\Omega }(z):=
\sum_{k=1}^{A_{\Omega }} \func{Spec}_g \left(h_k^\Omega \right)(z)=
\sum_{k=1}^{A_{\Omega }}\left\vert V_{g}h_{k}^{\Omega
}(z)\right\vert ^{2},\qquad z\in {\mathbb{R}^{2d}}\text{.}
\label{accumulated}
\end{equation}
(When the underlying window $g$ is clear from the context, we write
$\rho_{\Omega }$ instead of $\rho _{g,\Omega }$.)
\end{definition}

The idea of adding several uncorrelated spectrograms is not new. It is a basic tool in
spectral estimation \cite{BB}, it has been applied to the analysis of brain
signals \cite{Brain2014} and it is also an important step in the recent
high-resolution time-frequency algorithm ConceFT \cite{ConceFT}. The accumulated spectrogram has also been investigated 
in non-Euclidean contexts \cite{gh1} - see also \cite{hut}. Numerical
experiments show that when $N$ is close to the critical value $\approx \left\vert \Omega \right\vert $, the sum of the
$N$ spectrograms that are most concentrated on $\Omega$ almost exhaust the domain $\Omega$;
i.e., the accumulated spectrogram looks approximately like $1_\Omega$.
The following result from \cite{AGR16} provides a rigorous formulation of this observation.

\begin{theorem}
\cite[Theorem 1.3]{AGR16} Let $g\in L^{2}({\mathbb{R}^{d})}$, $\left\Vert
g\right\Vert _{2}=1$, and let $\Omega \subset {\mathbb{R}^{2d}}$ be compact.
Then, in $L^{1}({\mathbb{R}^{2d})}$,
\begin{equation}
\lim_{R\rightarrow \infty }\rho _{R\cdot \Omega }(R \cdot)\rightarrow 1_{\Omega }\text{.}
\label{asymlim}
\end{equation}
\end{theorem}
This shows that a large domain - $R \Omega$, with $R \gg 1$ -
can be approximated, by sensing the intensity of the corresponding
eigenspectrograms, \emph{without additional knowledge of the window}.
However, such a conclusion is only asymptotic. In practice, one needs a quantitative approximation
estimate. In other words, we want to measure the rate of
convergence to the limit in (\ref{asymlim}). We assume that $g$ satisfies
the following time-frequency concentration condition:
\begin{equation}
\lVert g\rVert _{M^{\ast }}^{2}:=\int_{\mathbb{R}^{2d}}\left\vert
z\right\vert \left\vert V_{g}g(z)\right\vert ^{2}dz<+\infty,
\label{eq_modstar}
\end{equation}
and let ${M^{\ast }}({\mathbb{R}^{d}})$ denote the class of all $L^{2}({\mathbb{R}^{d}})$ functions
satisfying \eqref{eq_modstar}. The following has
been obtained in \cite{AGR16}.

\begin{theorem}[\cite{AGR16}] Let $g\in {M^{\ast }}({\mathbb{R}^{d}})$ with $\left\Vert g\right\Vert
_{2}=1$ and $\Omega \subset \mathbb{R}^{2d}$ a
compact set with finite perimeter. Then
\begin{equation}
\left\Vert \rho _{g,\Omega }-1_{\Omega }\right\Vert _{L^{1}({\mathbb{R}^{2d}})}\leq
C_{g}\sqrt{\left\vert \partial {\Omega }\right\vert }\sqrt{{|\Omega |}}\text{,}  \label{estTAMS}
\end{equation}
where $\left\vert \partial {\Omega }\right\vert $ is the perimeter of $\Omega $ and $C_{g}$ is a
constant that only depends on $g$.
\end{theorem}

The main result of this article is the following improvement of (\ref{estTAMS}).

\begin{theorem}
\label{th_one_point}
Let $g\in {M^{\ast }}({\mathbb{R}^{d}})$ with $\left\Vert g\right\Vert
_{2}=1 $ and $\Omega \subset \mathbb{R}^{2d}$ a compact set with finite
perimeter and $\left\vert \partial {\Omega }\right\vert \geq 1$. Then
\begin{equation*}
\left\Vert \rho _{g,\Omega }-1_{\Omega }\right\Vert _{L^{1}({\mathbb{R}^{2d}})}\leq C_{g}\left\vert
\partial {\Omega }\right\vert \text{.}
\end{equation*}
\end{theorem}

Similar $L^{1}$ bounds have applications in signal analysis, namely in multi-taper
stabilization of power spectrum estimation \cite{AR}. Besides being technically
related to the main results \cite{AR}, we expect the present results to be instrumental
in non-stationary multi-taper estimation \cite{BB}.

\subsection{Phaseless approximation of time-frequency filters}
The accumulated spectrogram is defined in terms of the \emph{absolute value}
of the short-time Fourier transform of the most significant eigenfunctions $V_g h^\Omega_1, \ldots,
V_g h^\Omega_{A_\Omega}$, and is thus
related to the \emph{phase retrieval} problem - see e.g. \cite{jam, Jam07, babocaed09, phase}. Indeed,
for a Schwartz-class window $g$, time-frequency localization operators (with
general symbols as in \eqref{eq_tf_loc}) satisfy a trace-norm estimate
\begin{align*}
\norm{H_m}_{S^1} \lesssim \norm{m}_{L^1},
\end{align*}
where $S^1$ denotes the trace norm - see e.g. \cite{cogr03, teo}. Therefore, the estimate in Theorem
\ref{th_one_point} implies the spectral error bound
\begin{align*}
\norm{H_{\rho_{\Omega}} - H_\Omega}_{S^1}
=
\norm{H_{\rho_{\Omega}-1_\Omega}}_{S^1}
\lesssim \norm{\rho_{\Omega}-1_\Omega}_{L^1} 
\lesssim \abs{\partial \Omega}.
\end{align*}
In this way, the \emph{absolute values} of the short-time Fourier transforms of the eigenfunctions
$\{h^\Omega_k: k \geq 1\}$ lead to a spectral approximation of the operator $H_{\Omega}$.

\subsection{Sharpness of the estimates}
We establish the sharpness of the estimate in Theorem \ref{th_one_point} by testing it on
the family of all Euclidean balls.
\begin{theorem}
\label{th_sharp}
Let $g \in L^2(\Rdst)$ have norm 1 and let $B_R \subset \Rtdst$ be the ball of radius $R>0$
centered at the origin. Then there exist constants $C, C'$ such that
\begin{equation}
\label{eq_Cs}
C R^{2d-1}
\leq \left\Vert \rho _{g,{B_R} }-1_{{B_R} }\right\Vert _{L^{1}({\mathbb{R}^{2d}})}
\leq C' R^{2d-1}, \qquad R>0.
\end{equation}
\end{theorem}
To compare, the estimate given by (\ref{estTAMS}) is of the order $R^{2d-1/2}$.
See Figures \ref{fig_1} and \ref{fig_2} for an illustration of
Theorem \ref{th_sharp}, and \cite{AGR16} for numerical examples with other domains.
Results in the spirit of Theorem \ref{th_one_point} are also available in the context
of spaces of weighted analytic functions (see e.g. \cite{tian}) and the corresponding bounds match the ones
in Theorem \ref{th_one_point}.

\section{Example: Gaussian windows}
\label{sec_ex}
Let us consider the spectral problem associated with an operator of the form
(\ref{localization}), and symbol $1_{B_R}(z)$, the indication function of a disk
$B_R \subseteq \Rst^2$. Let
\begin{equation}
\newHz(t):=2^{1/4}e^{-\pi t^{2}},\text{ \ \ }t\in {\mathbb{R}}
\label{gaussian}
\end{equation}
be the one-dimensional Gaussian, normalized in $L^{2}$. Daubechies (on the signal side \cite{da88})
and Seip (directly on the phase space \cite{Seip91}) calculated the eigenfunctions and eigenvalues
of the time-frequency localization operator with window $\newHz$ and domain $\Omega =B_R$. For all
$R>0$,
the eigenfunctions of $H_{B_R}$ are the
Hermite functions:
\begin{equation*}
\newHk(t) = \frac{2^{1/4}}{\sqrt{k!}}\left(\frac{-1}{2\sqrt{\pi}}\right)^k
e^{\pi t^2} \frac{d^k}{dt^k}\left(e^{-2\pi t^2}\right), \qquad k \geq 0.
\end{equation*}
Writing $z:=x+i\xi \in {\mathbb{C}}$, the short-time Fourier transform of
the Hermite functions with respect to $h_{0}$ is
\begin{equation*}
V_{\newHz}\newHk(z)=e^{\pi ix\xi }\left( \frac{\pi ^{k}}{k!}\right)^{1/2}
\overline{z}^{k}e^{-\pi \left\vert z\right\vert ^{2}/2},\qquad k\geq 0\text{.}
\end{equation*}
Hence the corresponding spectrograms are
\begin{equation*}
\left\vert V_{\newHz}\newHk(z)\right\vert ^{2}=\frac{\pi ^{k}}{k!}\left\vert
z\right\vert ^{2k}e^{-\pi \left\vert z\right\vert ^{2}},\qquad k\geq 0\text{.}
\end{equation*}
Now, we have $\left\vert \Omega \right\vert =\left\vert B_R
\right\vert =\pi $ $R^{2}$. Consequently, in Definition \ref{def}, we can take $N=\lceil \pi $
$R^{2}\rceil $. The resulting
accumulated spectrogram is
\begin{equation}
\rho _{\newHz,B_R}(z)=\sum_{k=0}^{\lceil \pi R^{2}\rceil -1}
\frac{\pi ^{k}}{k!}\left\vert z\right\vert ^{2k}e^{-\pi \left\vert
z\right\vert ^{2}}\text{.}  \label{acginibre}
\end{equation}
Theorem \ref{th_one_point} says that, as $R\longrightarrow +\infty $,
\begin{equation}
\rho _{\newHz,B_R}(Rz)=\sum_{k=0}^{\lceil \pi R^{2}\rceil -1}
\frac{\pi ^{k}}{k!}R^{2k}\left\vert z\right\vert ^{2k}e^{-\pi \left\vert
Rz\right\vert ^{2}}\longrightarrow 1_{B_1}(z),\mbox{ in }
L^{1}({\mathbb{C}},dz)\text{,}  \label{limit}
\end{equation}
and gives a convergence rate of $O(1/R)$.
Thus one recovers the indicator function of the disk. Theorem \ref{th_one_point} provides us
with much more information: since the limit function $1_{\Omega }$ does not depend on the
window $g$, if one replaces the Gaussian (\ref{gaussian}) by \emph{an
arbitrary} $g\in {M^{\ast }}({\mathbb{R}^{d}})$ with $\left\Vert
g\right\Vert _{2}=1$, the limit (\ref{limit}) would give the same
result, and moreover the convergence rate is still $O(1/R)$. This gives a method to evaluate limits
of the type (\ref{limit}) in situations where one has no explicit formulas for $V_{g}h^\Omega_{k}(z)$ or
when the formulas are too complex to be analyzed directly.

\begin{figure}
\centering
\subfigure[The accumulated spectrogram corresponding to a circular domain of area $\approx$ 28.]{
\includegraphics[scale=0.4]{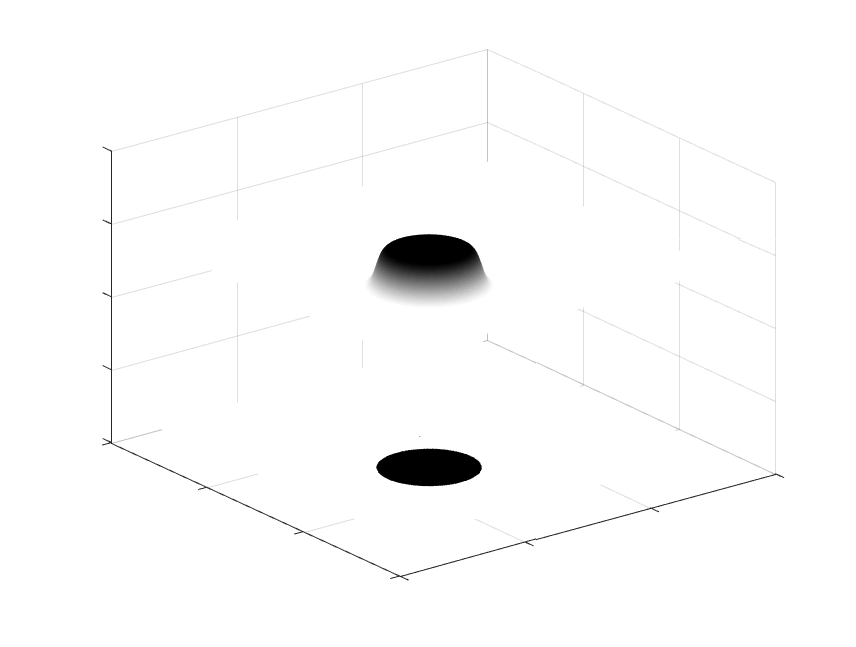}
\label{fig_dom}
}
\subfigure[A plot of the eigenvalues of the corresponding TF localization operator]{
\includegraphics[scale=0.4]{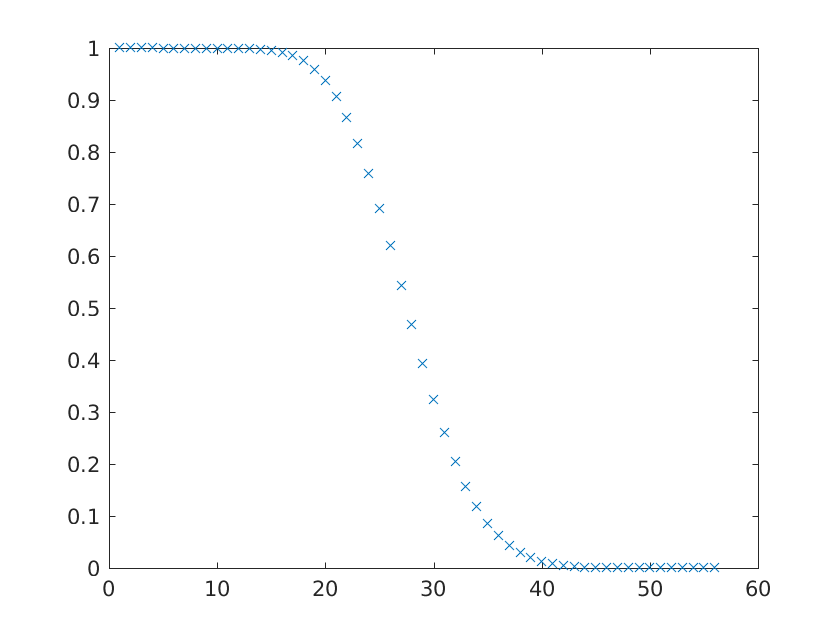}
}
\caption{An illustration of Theorem \ref{th_sharp} with a Gaussian window.}
\label{fig_1}
\end{figure}

\begin{figure}
\centering
\includegraphics[scale=0.7]{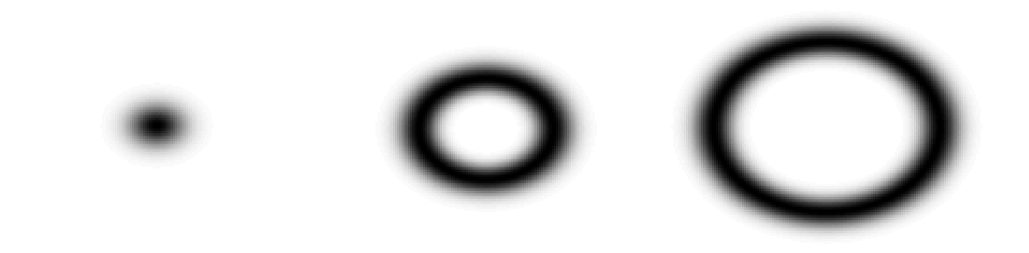}
\caption{Some of the eigenspectrograms corresponding to Figure \ref{fig_1}.}
\label{fig_2}
\end{figure}

\section{Some tools}
\subsection{Regularization by convolution}
A function $f\in L^{1}(\mathbb{R}^{d})$ is said to have bounded variation if
its distributional partial derivatives are finite Radon measures. We let
$\mathit{BV}({\mathbb{R}^{d}})$ denote the space of functions of bounded
variation. The variation of $f$ is defined as
\begin{equation*}
\func{Var}(f):=\sup \left\{ \int_{{\mathbb{R}^{d}}}f(x)\func{div}\phi (x)dx:\phi
\in C_{c}^{1}(\mathbb{R}^{d},\mathbb{R}^{d}),\left\vert \phi (x)\right\vert
_{2}\leq 1\right\} \text{,}
\end{equation*}
where $C_{c}^{1}(\mathbb{R}^{d},\mathbb{R}^{d})$ denotes the class of
compactly supported $C^{1}$-vector fields and $\func{div}$ is the divergence
operator. If $f$ is continuously differentiable and has integrable derivatives,
then $f\in \mathit{BV}({\mathbb{R}^{d}})$ and
\begin{equation*}
\func{Var}(f)=\int_{{\mathbb{R}^{d}}}\left\vert \nabla f(x)\right\vert _{2}dx\text{.}
\end{equation*}
A set $\Omega \subseteq {\mathbb{R}^{d}}$ is said to have \emph{finite
perimeter} if its indicator function $1_{\Omega }$ is of bounded variation.
In this case, the corresponding perimeter is defined as
\begin{equation*}
\abs{\partial \Omega }=\func{Var}(1_{\Omega })\text{.}
\end{equation*}
Every compact set $\Omega \subseteq {\mathbb{R}^{d}}$ with smooth boundary
has a finite perimeter and its perimeter is the $(d-1)$-dimensional surface
measure of its topological boundary (see \cite[Chapter 5]{evga92} for a
detailed account of these topics). The following lemma quantifies the error
introduced by convolution regularization. (For a proof, see \cite[Lemma 3.2]{AGR16}.)

\begin{lemma}
\label{lemma_var} Let $f\in \mathit{BV}({\mathbb{R}^{d}})$ and $\varphi \in
{L^{1}({\mathbb{R}^{d}})}$ with $\int \varphi =1$.
Then
\begin{equation*}
\lVert f\ast \varphi -f\rVert _{L^{1}({\mathbb{R}^{d}})}\leq
\func{Var}(f)\int_{\mathbb{R}^{d}}\left\vert x\right\vert _{2}\left\vert \varphi
(x)\right\vert dx\text{.}
\end{equation*}
In particular, for a set of finite perimeter $\Omega $:
\begin{equation*}
\lVert 1_{\Omega }\ast \varphi -1_{\Omega }\rVert _{L^{1}({\mathbb{R}^{d}})}\leq \left\vert
\partial
{\Omega }\right\vert \int_{\mathbb{R}^{d}}\left\vert x\right\vert _{2}\left\vert \varphi
(x)\right\vert dx\text{.}
\end{equation*}
\end{lemma}

\subsection{Traces and disks}
\begin{proposition}
\label{prop_eig_2}
Let $g \in L^2(\Rdst)$ have norm 1 and let $B_R \subset \Rtdst$ be the ball of radius $R>0$
centered at the origin. Then there exist a constant $C>0$ such that
\begin{align}
\label{eq_conv}
\mathrm{trace}(H_{{B_R} }) - \mathrm{trace}(H^2_{B_R})
\geq C R^{2d-1}.
\end{align}
\end{proposition}
\begin{proof}
This result is a special case of the asymptotics for the so-called
\emph{plunge region} of the eigenvalues of time-frequency localization operators \cite{da88, dapa88, da90}
- that is, the region where the eigenvalues are away from both $0$ and $1$, cf. Figure \ref{fig_1}; 
see also \cite{rato93, rato94, herato94, herato97, defeno02}.
Concretely, \eqref{eq_conv}
follows, for example, by combining Proposition 4.2 and Lemma 4.3 in \cite{defeno02}. (The result in
\cite{defeno02} applies to certain families of domains that behave qualitatively like dilates of a
single domain; the constant $C$ depends on the size of $\abs{V_g g}$ near the origin.)
\end{proof}

\section{Proof of Theorem \ref{th_one_point}}

\textbf{Step 1}. A direct calculation shows that the traces of $H_{\Omega }$
and $H_{\Omega }^{2}$ are given by
\begin{align}
& \mathrm{trace}(H_{\Omega })={|\Omega |},  \label{eq_trace1} \\
& \mathrm{trace}(H_{\Omega }^{2})=\int_{\Omega }\int_{\Omega }\left\vert
V_{g}g(z-z^{\prime })\right\vert ^{2}dzdz^{\prime }\text{.}
\label{eq_trace2}
\end{align}
(See for example \cite[Lemma 2.1]{AGR16}.). Therefore,
\begin{equation*}
\begin{aligned}
0 \leq \mathrm{trace}(H_{\Omega })-\mathrm{trace}(H_{\Omega
}^{2})&=\int_{\Omega }1-\left( 1_{\Omega }\ast \left\vert V_{g}g\right\vert
^{2}\right) (z)\,dz \\
& \leq \lVert 1_{\Omega }\ast \left\vert V_{g}g\right\vert ^{2}-1_{\Omega
}\rVert _{L^{1}({\mathbb{R}^{2d}})}\text{.}
\end{aligned}
\end{equation*}
Hence, by Lemma \ref{lemma_var} we conclude that
\begin{equation}
\label{eq_dev_1}
0\leq \mathrm{trace}(H_{\Omega })-\mathrm{trace}(H_{\Omega }^{2})\leq \lVert
g\rVert _{M^{\ast }}^{2}\left\vert \partial {\Omega }\right\vert \text{.}
\end{equation}
\textbf{Step 2}. Since
\begin{eqnarray*}
\mathrm{trace}(H_{\Omega }) &=&\sum_{k\geq 1}{\lambda _{k}^{\Omega }}, \\
\mathrm{trace}(H_{\Omega }^{2}) &=&\sum_{k\geq 1}{(\lambda _{k}^{\Omega }})^{2},
\end{eqnarray*}
we have
\begin{align*}
& \mathrm{trace}(H_{\Omega })-\mathrm{trace}(H_{\Omega }^{2})=\sum_{k\geq 1}{\lambda _{k}^{\Omega
}}(1-{\lambda _{k}^{\Omega }}) \\
& \qquad  =\sum_{k=1}^{A_{\Omega }}{\lambda _{k}^{\Omega }}(1-{\lambda
_{k}^{\Omega }})+\sum_{k>A_{\Omega }}{\lambda _{k}^{\Omega }}(1-{\lambda
_{k}^{\Omega }}) \\
& \qquad \geq \lambda _{A_{\Omega }}\sum_{k=1}^{A_{\Omega }}(1-{\lambda
_{k}^{\Omega }})+(1-\lambda _{A_{\Omega }}^{\Omega })\sum_{k>A_{\Omega }}{\lambda _{k}^{\Omega }} \\
& \qquad =\lambda _{A_{\Omega }}A_{\Omega }-\lambda _{A_{\Omega
}}\sum_{k=1}^{A_{\Omega }}{\lambda _{k}^{\Omega }}+(1-\lambda _{A_{\Omega
}}^{\Omega })
\left(\left\vert \Omega \right\vert -\sum_{k=1}^{A_{\Omega }}{\lambda _{k}^{\Omega }}\right)
\\
& \qquad =\lambda _{A_{\Omega }}A_{\Omega }+\left\vert \Omega \right\vert
(1-\lambda _{A_{\Omega }}^{\Omega })-\sum_{k=1}^{A_{\Omega }}{\lambda
_{k}^{\Omega }} \\
& \qquad =\left\vert \Omega \right\vert -\sum_{k=1}^{A_{\Omega }}{\lambda
_{k}^{\Omega }}+\lambda _{A_{\Omega }}^{\Omega }(A_{\Omega }-\left\vert
\Omega \right\vert ) \\
& \qquad \geq \left\vert \Omega \right\vert -\sum_{k=1}^{A_{\Omega }}{\lambda _{k}^{\Omega }}.
\end{align*}
Therefore, using the estimate in \eqref{eq_dev_1}, we obtain
\begin{equation}
\label{eq_sum}
\left\vert \Omega \right\vert -\sum_{k=1}^{A_{\Omega }}{\lambda _{k}^{\Omega
}}\leq C_{g}\left\vert \partial {\Omega }\right\vert \text{.}
\end{equation}
\textbf{Step 3}. Since $0\leq \rho_\Omega (z)\leq 1$, we can estimate
\begin{equation}
\begin{aligned}
\int_{\Omega }\left\vert \rho_\Omega (z)-1_{\Omega }(z)\right\vert \,dz
&=\left\vert
\Omega \right\vert -\int_{\Omega }\rho_\Omega (z)\,dz
\\
&=\abs{\Omega}-\sum_{k=1}^{A_\Omega} \int_{\Omega } \abs{V_g h^\Omega_k(z)}^2 dz
\\
&=\abs{\Omega}-\sum_{k=1}^{A_\Omega}
\ip{H_\Omega h^\Omega_k}{h^\Omega_k}
\\
&=\left\vert \Omega \right\vert
-\sum_{k=1}^{A_{\Omega }}{\lambda _{k}^{\Omega }}.
\end{aligned}
\label{eq_1}
\end{equation}
Similarly,
\begin{align}
\label{eq_2}
\int_{{\mathbb{R}^{2d}}\setminus \Omega }\left\vert \rho_\Omega (z)-1_{\Omega
}(z)\right\vert \,dz &=\int_{{\mathbb{R}^{2d}}\setminus \Omega }\rho_\Omega (z)\,dz
=\sum_{k=1}^{A_{\Omega }}(1-{\lambda _{k}^{\Omega }})
\\
\label{eq_3}
&=
A_{\Omega}-\sum_{k=1}^{A_{\Omega }}{\lambda _{k}^{\Omega }}
\\
\nonumber
&\leq 1+\left\vert \Omega
\right\vert -\sum_{k=1}^{A_{\Omega }}{\lambda _{k}^{\Omega }}.
\end{align}
Using \eqref{eq_sum} and the fact that $\left\vert \partial {\Omega }
\right\vert \geq 1$ we obtain the desired conclusion:
\begin{equation*}
\int_{{\mathbb{R}^{d}}}\left\vert \rho_\Omega (z)-1_{\Omega }(z)\right\vert
\,dz\leq 1+2\left( \left\vert \Omega \right\vert -\sum_{k=1}^{A_{\Omega }}{\lambda _{k}^{\Omega
}}\right) \leq C_{g}\left\vert \partial {\Omega }\right\vert \text{.}
\end{equation*}

\section{Proof of Theorem \ref{th_sharp}}
In view of Theorem \ref{th_one_point}, we only need to prove the lower bound
in \eqref{eq_Cs}.
Let $\Omega \subseteq \Rtdst$ be a compact domain with finite perimeter.
Using \eqref{eq_1}, \eqref{eq_2} and \eqref{eq_3} we obtain
\begin{align*}
\int_{\mathbb{R}^{2d} }\left\vert \rho_\Omega (z)-1_{\Omega }(z)\right\vert \,dz
&=\left\vert \Omega \right\vert
+ A_\Omega
-2 \sum_{k=1}^{A_{\Omega }}{\lambda _{k}^{\Omega }}
\\
& =
\left( A_\Omega - \sum_{k=1}^{A_{\Omega }}{\lambda _{k}^{\Omega }} \right)
+
\left(
\mathrm{trace}(H_{\Omega }) - \sum_{k=1}^{A_{\Omega }}{\lambda _{k}^{\Omega }}
\right)
\\
& = \sum_{k=1}^{A_\Omega} (1-\lambda^\Omega_k) +
\sum_{k>A_\Omega} \lambda^\Omega_k
\\
& \geq \sum_{k=1}^{A_\Omega} \lambda^\Omega_k (1-\lambda^\Omega_k) +
\sum_{k>A_\Omega} \lambda^\Omega_k (1-\lambda^\Omega_k)
\\
& = \mathrm{trace}(H_{\Omega }) -
\mathrm{trace}(H^2_{\Omega }).
\end{align*}
We apply these calculations to $\Omega=B_R$ to obtain
\begin{align}
\label{eq_a}
 \int_{\mathbb{R}^{2d} }\left\vert \rho_{{B_R}} (z)-1_{{B_R} }(z)\right\vert \,dz
 \geq \mathrm{trace}(H_{{B_R} }) - \mathrm{trace}(H^2_{{B_R} }).
\end{align}
Therefore, by Proposition \ref{prop_eig_2},
\begin{align}
 \int_{\mathbb{R}^{2d} }\left\vert \rho_{{B_R}} (z)-1_{{B_R} }(z)\right\vert \,dz
 \gtrsim R^{2d-1},
\end{align}
as desired.

\section{Acknowledgements}
Part of the research for this article was conducted while L. D. A. and J. L. R. visited
the Program in Applied and Computational Mathematics at Princeton
University. They thank PACM and in particular Prof. Amit Singer
for their kind hospitality. The figures were prepared with the Large Time-Frequency Analysis
Toolbox (LTFAT) \cite{ltfat1,ltfat2}.

\bibliographystyle{abbrv}

\end{document}